\documentclass[11pt]{amsart}

\usepackage[utf8]{inputenc}

\usepackage{amsmath,amssymb,amstext,amsthm} 
\usepackage[
final]{showkeys}
\usepackage[american]{babel}
\usepackage{hyperref}
\usepackage[obeyDraft]{todonotes} 

\usepackage{amsfonts}
\usepackage{amssymb}
\usepackage{amsthm}
\usepackage{amsmath}
\usepackage{graphicx}
\usepackage{url}
\usepackage[all]{xy}
\usepackage{lscape}
\usepackage{mathtools}
\usepackage{color}
\usepackage{tikz}
\usepackage{float}
\usepackage{comment}
\usetikzlibrary{shapes.geometric}
\usetikzlibrary{shapes}
\usetikzlibrary{arrows}
\usetikzlibrary{decorations.markings}
\usetikzlibrary{tqft}
\usetikzlibrary{positioning}
\usetikzlibrary{patterns}
\usepackage{dsfont}
\usepackage{tikz-cd}
\usepackage{enumerate}
\usepackage{fancyhdr}
\usepackage{enumitem}

\usepackage{cleveref}

\makeatletter
\renewcommand\subsection{\@startsection{subsection}{2}%
  \z@{-.5\linespacing\@plus-.7\linespacing}{.5\linespacing}%
  {\normalfont\scshape}}

\usepackage{amsthm}

\theoremstyle{definition}
\newtheorem{defn}{Definition}[section]

\theoremstyle{remark}
\newtheorem{rem}[defn]{Remark}

\theoremstyle{plain}
\newtheorem{thm}[defn]{Theorem}
\newtheorem{prop}[defn]{Proposition}
\newtheorem{lem}[defn]{Lemma}
\newtheorem{cor}[defn]{Corollary}

\newtheorem{quest}[defn]{Question}

\usepackage{aliascnt}

\usepackage{xspace}

\usepackage{xcolor}
\usepackage{hyperref}

\newcommand{\ie}{i.\,e.~}

\newcommand{\eg}{e.\,g.~}
 
\newcommand{\hfk}{\widehat{\operatorname{HFK}}}

\usepackage{etoolbox}

\docsvlist{A,B,C,D,E,F,G,H,I,J,K,L,M,N,O,P,Q,R,S,T,U,V,W,X,Y,Z}

\docsvlist{N,Z,C,Q,R,F,A,W,M}

\docsvlist{A,B,C,D,E,F,G,H,I,J,K,L,M,N,O,P,Q,R,S,T,U,V,W,X,Y,a,b,c,d,e,f,h,i,j,k,m,n,o,p,q,r,s,t,u,v,w,x,y,z}

\docsvlist{alt,dalt,dim,Hom,id,Kh,lk,rank,rk,tb,TB}

\newcommand{\sigij}{\sigma_{i,j}}

\usepackage{caption}
\usepackage{subcaption}

\usepackage{chngcntr} 

\crefname{thm}{theorem}{theorems}
\newcounter{nparcount}

\newenvironment{sizeddisplay}[1]
 {\par\nopagebreak#1\noindent\ignorespaces}
 {\nopagebreak\ignorespacesafterend}

\title{Non-complex cobordisms between quasipositive knots}
\author{Maciej Borodzik}
\address{Institute of Mathematics, University of Warsaw, ul. Banacha 2, 02-096, Warsaw, Poland.}
\email{mcboro@mimuw.edu.pl}
\author{Paula Truöl}
\address{Max Planck Institute for Mathematics, Vivatsgasse 7, 53111 Bonn, Germany.}
\email{paulagtruoel@gmail.com}

\begin{document}

\def\subjclassname{\textup{2020} Mathematics Subject Classification}
\expandafter\let\csname subjclassname@1991\endcsname=\subjclassname

\begin{abstract}
We show that for every genus $g \geq 0$, there exist quasipositive knots $K_0^g$ and $K_1^g$ such that there is a cobordism of genus $g=|g_4(K_1^g)-g_4(K_0^g)|$ between~$K_0^g$ and $K_1^g$, but there is no ribbon cobordism of genus $g$ in either direction and thus no complex cobordism between these two knots. This gives a negative answer to a question posed by Feller in 2016. 
\end{abstract}
\maketitle

\section{Introduction}
Let $K_0,K_1\subset S^3$ be knots. A~\emph{cobordism} of genus $g$ between $K_0$ and~$K_1$ is an oriented, connected, compact, smoothly embedded surface $\Sigma\subset S^3\times[0,1]$ with genus $g(\Sigma)=g$ and boundary $\partial\Sigma=K_1\times\{1\}\sqcup -K_0\times\{0\}$. In this paper we work exclusively in the smooth category. 
A~cobordism is \emph{ribbon} if~$\Sigma$ is embedded in such a way that the projection $S^3\times[0,1]\to[0,1]$ restricted to $\Sigma$ is a Morse function without local maxima. A cobordism is \emph{optimal} if $g(\Sigma)=|g_4(K_1)-g_4(K_0)|$, where $g_4(K_i)$ denotes the \emph{$4$-genus} of~$K_i$, $i \in \{0,1\}$, \ie the minimal genus of a cobordism between $K_i$ and the unknot. A \emph{concordance}, that is a cobordism with $g(\Sigma)=0$, is optimal by definition.

A motivating example for the study of optimal ribbon cobordisms comes from the theory of complex plane curves. Suppose $C\subset\C^2$ is a non-singular complex curve (\ie the zero set of a square-free polynomial in $\C[x,y]$) and~$S_0, S_1$ are two $3$-spheres in $\C^2$ with a common center, each intersecting~$C$ transversally. Set $K_0=C\cap S_0$ and $K_1=C\cap S_1$, and let $\Sigma$ be the part of~$C$ between $S_0$ and $S_1$. 
A classical argument using the maximum principle for the distance function restricted to $C$ shows that $\Sigma$ is a ribbon cobordism between $K_0$ and $K_1$ (see \eg \cite[Lemma 2.6]{borodzik_morse_theory_curves}).
The resolution of the Thom conjecture by Kronheimer--Mrowka~\cite{kronheimermrowka,rudolphslice}
implies that $g(\Sigma)=|g_4(K_1)-g_4(K_0)|$. 
A cobordism $\Sigma$ obtained as above from a non-singular complex curve is called \emph{complex}. 

Not every knot arises as a transverse intersection of a non-singular complex curve with a sphere in $\C^2$. Those that do are called \emph{quasipositive}, see~\cite{Rudolph_1983,boileau_orevkov_2001}. According to this work of Rudolph and Boileau--Orevkov, quasipositive knots can be characterized as the closures of quasipositive braids (see more below). It is interesting to compare optimal cobordisms between quasipositive knots and complex cobordisms. The following question was first asked by Feller.

\begin{quest}[{\cite{Feller_2016}}]\label{quest:Feller}
Are the two necessary conditions for the existence of a complex cobordism between two knots---the knots are quasipositive and there exists an optimal cobordism between them---sufficient?
\end{quest}
Our main result answers this question negatively.
\begin{thm}\label{thm:main_qp}
For every $g \geq 0$, there exist quasipositive knots $K_0^g$ and $K_1^g$ such that there is an optimal cobordism of genus $g$ between $K_0^g$ and $K_1^g$, but there is no complex cobordism between $K_0^g$ and $K_1^g$. 
\end{thm}
The proof, given in Section~\ref{sec:main_proofs}, shows that no optimal cobordism between $K_0^g$ and $K_1^g$ is ribbon. To prove this, we successfully exploit a recent obstruction to ribbon cobordisms by Livingston~\cite{livingston_crit_point_counts}, see Theorem~\ref{thm:livingston} below. 

A special class of quasipositive knots is given by \emph{strongly quasipositive} knots, which are closures of strongly quasipositive braids~\cite{rudolph_braidedsurfaces,rudolph_linkpolys}. 
A braid on $n$ strands is \emph{quasipositive} if it is a finite product of conjugates of the positive
Artin generators $\sigma_i$~\cite{artin_1925}, and \emph{strongly quasipositive} if the product consists only of certain conjugates of the $\sigma_i$, namely 
\begin{align*}
\sigij = \left(\sigma_i \cdots \sigma_{j-2} \right) \sigma_{j-1} \left(\sigma_i \cdots \sigma_{j-2} \right)^{-1} \qquad \text{for} \quad 1 \leq i < j \leq n.
\end{align*}
Within the class of quasipositive braids, strongly quasipositive braids are special because they come with an associated Seifert surface that realizes the minimal genus among Seifert surfaces for their closure~\cite{bennequin}.
In the context of (smooth) concordances between knots, it follows from Rudolph's slice--Bennequin inequality~\cite{rudolphslice} that not every knot is concordant to a quasipositive knot, which is contrary to the behavior in the topological category~\cite{borodzikFeller}.

Building on \cite{truoel_sqp} and using Heegaard Floer homology as an obstruction to ribbon concordances~\cite{Zemke}, we are able to show the following strengthening of Theorem~\ref{thm:main_qp} in the case where $g=0$. 

\begin{prop}\label{prop:sqp_conc}
There exist strongly quasipositive knots $K_0$ and $K_1$ such that there is a concordance between them, but no ribbon (and thus no complex) concordance in either direction.
\end{prop}

We conclude the introduction with a few open questions. Firstly, our obstruction to complex cobordisms obstructs ribbon cobordisms. 
\begin{quest}
Suppose that there is an optimal ribbon cobordism between quasipositive knots $K_0$ and $K_1$. Does it follow that there is a complex cobordism between $K_0$ and $K_1$?
\end{quest}
The next question arises from the observation that the examples we construct in the proof of Proposition~\ref{prop:sqp_conc} are not fibered.
\begin{quest}\label{quest:2}
Suppose that there is an optimal cobordism between strongly quasipositive, fibered knots $K_0$ and $K_1$. Is there a complex cobordism?
\end{quest}
Conjecturally, concordant, strongly quasipositive, fibered knots are isotopic \cite{Baker_2015}.
A possible negative answer to Question~\ref{quest:2} would be to find two concordant, strongly quasipositive, fibered knots with different Seifert forms. At present, any known construction of concordant strongly quasipositive knots, even if we drop the fiberedness condition, is between knots which share their Seifert form. 
\begin{quest}\label{quest:Seifert_form}
Do there exist two strongly quasipositive knots that are concordant, but have distinct Seifert forms? 
\end{quest}
A positive answer to Question~\ref{quest:Seifert_form} could place us in a position of applying Livingston's results to obstruct complex concordances between strongly quasipositive knots.
\begin{rem}
In this paper we use the terminology `complex cobordism' instead of `algebraic cobordism' of e.g. \cite{Feller_2016} to avoid a terminology clash. There is the notion of an algebraically slice knot, see \eg \cite{Chuck_algeb,Hedden-Kirk-Livingston_algebraic}, which means that the Seifert form of the knot is metabolic. The term `algebraically cobordant' is dangerously similar to `algebraically slice'. 
\end{rem}

Neither of the properties `optimal' or `ribbon' of a cobordism implies the other. Theorem~\ref{thm:main_qp} and Proposition~\ref{prop:sqp_conc} show that optimal cobordisms between (strongly) quasipositive knots are not necessarily ribbon. 
If we drop the condition of quasipositivity, examples are easier to obtain and well-known:  
Let $K$ denote the trefoil and $J$ the figure-eight knot. Then $K_0 = K \# -K$ and $K_1 = J \# -J$ are both slice, thus concordant. 
On the other hand, one can consider the coprime Alexander polynomials of $K_0$ and~$K_1$ to show that they are not ribbon concordant \cite{gordon,friedl_powell}.

To see that not every ribbon cobordism between strongly quasipositive knots is optimal, consider the following example. There is a ribbon cobordism~$\Sigma$ of genus one between the trefoil and the positive untwisted Whitehead double of the trefoil, which can be constructed by using two saddle moves; an intermediate stage of the cobordism is a positive Hopf link. Both knots are strongly quasipositive and have $4$-genus one \cite{rudolphslice}, so $\Sigma$ is not optimal. 

\subsubsection*{Organization.}
We recall Livingston's obstruction in Section~\ref{sec:chuck} and prove Theorem~\ref{thm:main_qp} in Section~\ref{sec:main_proofs}. The proof of Proposition~\ref{prop:sqp_conc} can be found in Section~\ref{sec:constr}, while the computational details of the relevant PD codes are discussed in Section~\ref{sec:details}.

\subsubsection*{Acknowledgements.} We are indebted to Hannah Turner for many fruitful discussions. 
We would also like to thank L\'eo B\'enard, Peter Feller, Daniel Galvin and Lukas Lewark for helpful discussions and comments. We would like to express our gratitude to Arunima Ray for her many comments on a draft version of the paper. We thank the anonymous referee for very helpful comments and suggestions, particularly the alternative proof of Lemma 2.2. Part of this project was carried out during the first author's visit to the Max Planck Institute for Mathematics in Bonn. The authors would like to thank the institute for its hospitality and support. The first author was covered by the Polish NCN Opus grant 2024/53/B/ST1/03470. 

\section{Background on double branched covers of knots}\label{sec:chuck}

\subsection{Classical results}\label{sec:branched_covers}
We will use the following elementary fact about double branched covers, see \eg \cite[Theorem 9.1 and Corollary 9.2]{Lickorish97}.

\begin{lem}\label{lem:hom_double_br_cover}
Let $\Sigma_2(K)$ denote the double branched cover of $S^3$ branched along $K$. Its first homology group $H_1(\Sigma_2(K); \Z)$ is an abelian group of finite order equal to the determinant $\det K$.
\end{lem}

\begin{lem}\label{lem:conn_sum}
Suppose $K_1,\dots,K_n$ are knots. Set $K=K_1\#\dots\#K_n$. Then
\[H_1(\Sigma_2(K);\Z)\cong \bigoplus_{i=1}^n H_1(\Sigma_2(K_i);\Z).\]
\end{lem}
\begin{proof}
The double branched cover of a connected sum of knots is the connected sum of the double branched covers of the knots. The fundamental group of~$\Sigma_2(K)$ therefore is the free product of the fundamental groups of the double branched covers $\Sigma_2(K_1), \dots, \Sigma_2(K_n)$. Since the first homology group with $\Z$ coefficients is the abelianization of the fundamental group, the result follows.
\end{proof}

As a corollary, we can prove the following.
\begin{prop}\label{prop:main_computation}
Let $p$ be a prime. If $K$ is a knot such that $p$ divides~$\det K$, then, as an $\F_p$-vector space, 
\[\dim_{\F_p} H_1(\Sigma_2(nK);\F_p)\ge n.\]
\end{prop}
\begin{proof}
Let $H=H_1(\Sigma_2(K);\Z)$.
If $p$ divides $\det K$, then by Lemma~\ref{lem:hom_double_br_cover}, $p$ divides the order of $H$. By the classification of finite abelian groups, $H$ can be written as $P\oplus Q$, where $P$ is a cyclic $p$-group. By Lemma~\ref{lem:conn_sum}, we obtain
\[H_1(\Sigma_2(nK);\Z)=nP\oplus nQ.\]
By the universal coefficient theorem  $H_1(\Sigma_2(nK);\F_p)=H_1(\Sigma_2(nK);\Z)\otimes_{\Z}\F_p$ contains a subgroup
$nP\otimes_{\Z}\F_p\cong \F_p^{\oplus n}$.
\end{proof}
The universal coefficient theorem argument used in the proof of Proposition~\ref{prop:main_computation} also implies the following well-known corollary.
\begin{cor}\label{cor:main_computation}
Suppose $K$ is a knot.
If a prime $p$ does not divide $\det K$, then $H_1(\Sigma_2(K);\F_p)=0$.
\end{cor}

\subsection{Livingston's obstruction}\label{sec:livingston}

In \cite{livingston_crit_point_counts}, Livingston studies for given knots $K_0$ and $K_1$, the set of all four-tuples $(g, c_0, c_1, c_2)$ of non-negative integers for which there is a cobordism~$\Sigma$ between $K_0$ and $K_1$ of genus $g(\Sigma)=g$ with $c_i(\Sigma) = c_i$ critical points of each index~$i \in \{0,1,2\}$. His obstruction to the existence of $\Sigma$ is in terms of homology groups of cyclic branched covers of $K_0$ and $K_1$. 
We use the following version of his result, noting that \cite{livingston_crit_point_counts} contains more general statements.
\begin{thm}[{\cite[Corollary 5.4]{livingston_crit_point_counts}}]\label{thm:livingston}
    Let $\Sigma$ be a cobordism between knots $K_0$ and $K_1$. 
    Then for all odd primes $p$, we have
    \begin{align*}
        c_0 (\Sigma) &\geq \frac{\beta_1(\Sigma_2(K_1), \F_p)- \beta_1(\Sigma_2(K_0), \F_p)}{2}-g(\Sigma) \qquad \text{and}\\
        c_2 (\Sigma) &\geq \frac{\beta_1(\Sigma_2(K_0), \F_p)- \beta_1(\Sigma_2(K_1), \F_p)}{2}-g(\Sigma),
    \end{align*}
    where for $i \in \{0,1\}$, $\beta_1(\Sigma_2(K_i), \F_p)$ is the dimension of $H_1(\Sigma_2(K_i); \F_p)$ as an~$\F_p$-vector space. 
\end{thm}

\section{Proof of \texorpdfstring{Theorem~\ref{thm:main_qp}}{Theorem 1.2}}\label{sec:main_proofs}
Since complex cobordisms are necessarily ribbon (see the introduction), we will focus on the construction of non-ribbon cobordisms. Theorem~\ref{thm:main_qp} will be deduced from the following more general statement.
\begin{prop}\label{prop:main_ribbon}
Let $K_0,K_1$ be two quasipositive knots that admit a cobordism of genus $g \geq 0$. Then there exist quasipositive knots $K_0^\prime$ and $K_1^\prime$ such that $K_0^\prime$ is concordant to $K_0$, $K_1^\prime$ is concordant to $K_1$, and there exists a cobordism of genus $g$ between $K_0^\prime$ and $K_1^\prime$. Moreover, no cobordism of genus~$g$ between $K_0^\prime$ and $K_1^\prime$ can be ribbon in any direction. 
\end{prop}
\begin{proof}
    Fix $g \geq 0$. Let $J_0$ and $J_1$ denote the 
    mirrors of the knots $8_{20}$ and $11n_{50}$, respectively, from the Rolfsen and Hoste-Thistlethwaite knot table.
    From Knotinfo~\cite{knotinfo} we read off the following properties of $J_0$ and $J_1$:
    \begin{enumerate}[label=(P-\arabic*)]
    \item $\det J_0=9$, $\det J_1=25$;\label{item:Pdet}
    \item $J_0$ and $J_1$ are slice; \label{item:Pslice}
    \item $J_0$ and $J_1$ are quasipositive.\label{item:Pquasi}
    \end{enumerate}
    Choose an integer $m>0$, of which we will later make a sufficiently large choice.
    Denote by $mJ_0$ and $mJ_1$ the knots obtained as connected sums of~$m$ copies of~$J_0$ and~$J_1$, respectively, and define
    \begin{align*}
    K_0^\prime=K_0\# mJ_0,\quad  K_1^\prime=K_1\# mJ_1.
    \end{align*}
    Two properties of $K_0^\prime$ and $K_1^\prime$ are immediate:
    \begin{itemize}
    \item $K_0^\prime$ and $K_1^\prime$ are quasipositive because $J_0$ and $J_1$ are by \ref{item:Pquasi}, and a connected sum of quasipositive knots is quasipositive. In fact, a connected sum of knots is quasipositive if and only if both summands are quasipositive~\cite{orevkov}. The direction we use is easy to see using the characterization of quasipositive knots as closures of quasipositive braids~\cite{Rudolph_1983,boileau_orevkov_2001}.
    \item $K_0^\prime$ is concordant to $K_0$, and $K_1'$ is concordant to $K_1$, because $J_0$ and~$J_1$ are slice by \ref{item:Pslice}.
    \end{itemize}
    To prove that there is no ribbon cobordism of genus $g$ between $K_0^\prime$ and $K_1^\prime$ in either direction, for $i\in \{0,1\}$ and a prime $p$, denote by $d_{ip}$ the dimension
    \[d_{ip}=\dim_{\F_p}H_1(\Sigma_2(K_i);\F_p).\]
    By Proposition~\ref{prop:main_computation}, Corollary~\ref{cor:main_computation} and \ref{item:Pdet}, we have
    \begin{align*}
    \dim_{\F_3} H_1(\Sigma_2(mJ_0);\F_3)&\ge m,\quad \dim_{\F_5} H_1(\Sigma_2(mJ_0);\F_5)=0,\\
    \dim_{\F_3} H_1(\Sigma_2(mJ_1);\F_3)&=0, \quad \dim_{\F_5} H_1(\Sigma_2(mJ_1);\F_5)\ge m.
    \end{align*}
    Hence by Lemma~\ref{lem:conn_sum} we obtain
    \begin{align*}
    \begin{split}
    \dim_{\F_3} H_1(\Sigma_2(K_0^\prime);\F_3)-\dim_{\F_3} H_1(\Sigma_2(K_1^\prime);\F_3)&\ge m+d_{03}-d_{13} \quad \text{and}\\
    \dim_{\F_5} H_1(\Sigma_2(K_1^\prime);\F_5)-\dim_{\F_5} H_1(\Sigma_2(K_0^\prime);\F_5)&\ge m+d_{15}-d_{05}.
    \end{split}
    \end{align*}
    Now, if we chose $m$ sufficiently large, say $m>2(g+d_{03}+d_{13}+d_{05}+d_{15})$, from Theorem~\ref{thm:livingston} we obtain that a cobordism of genus $g$ between $K_0^\prime$ and~$K_1^\prime$ must necessarily fulfill $c_0,c_2>0$. Therefore such a cobordism cannot be ribbon in either direction.
\end{proof}

We are now ready to prove Theorem~\ref{thm:main_qp}.

\begin{proof}[{Proof of Theorem~\ref{thm:main_qp}}]
  Suppose $K_0$ and $K_1$ are two quasipositive knots which admit an optimal cobordism of genus $g$. We can \eg take the torus knots $K_0=T_{2,3}$ and $K_1=T_{2,2g+3}$. The knots $K_0^\prime$ and $K_1^\prime$ from Proposition~\ref{prop:main_ribbon} then have genus $g_4(K_0^\prime)=g_4(K_0)=1$ and~$g_4(K_1^\prime)=g_4(K_1)=g+1$ \cite{kronheimermrowka}. In particular, a cobordism of genus $g$ between $K_0^\prime$ and $K_1^\prime$ is optimal. By Proposition~\ref{prop:main_ribbon}, there is a cobordism of genus $g$ between $K_0^\prime$ and $K_1^\prime$. But such a cobordism has positive $c_0$ and $c_2$, so 
    it cannot be ribbon in any direction.
\end{proof}

We conclude this section with a remark about the $g=0$ case of Theorem~\ref{thm:main_qp}.
Although the knots $K_0^\prime$ and $K_1^\prime$ have $g_4=1$, they arise as non-trivial connected sums and in particular have a Seifert genus strictly larger than~$1$.

However, we note that for this $g=0$ case of Theorem~\ref{thm:main_qp} we could have used the prime knots $J_0$ and $J_1$ from the proof of Proposition~\ref{prop:main_ribbon} directly. Indeed, by \ref{item:Pslice} and \ref {item:Pquasi}, $J_0$ and~$J_1$ are quasipositive and slice, so there is a cobordism of genus $0$ between them. \ref{item:Pdet} and Theorem~\ref{thm:livingston} can be used again to show that there is no ribbon, and thus no complex cobordism, between~$J_0$ and~$J_1$. Alternatively, one could use \ref{item:Pdet} and a result of Gilmer~\cite{gilmer}, which implies that the determinants of $J_0$ and $J_1$ would need to be divisible by each other if there were ribbon concordances between them.

\section{Construction of strongly quasipositive knots \texorpdfstring{$K_0$}{K0} and \texorpdfstring{$K_1$}{K1}}\label{sec:constr}

\subsection{Proof of \texorpdfstring{Proposition~\ref{prop:sqp_conc}}{Proposition 1.3}}\label{sec:sqp}

Recall the statement of Proposition~\ref{prop:sqp_conc}: 
There exist strongly quasipositive knots $K_0$ and $K_1$ such that there is a concordance between them, but no ribbon (and thus no complex) concordance in either direction.

\begin{proof}[{Proof of Proposition~\ref{prop:sqp_conc}}] The proof is based on a construction of the second author~\cite{truoel_sqp} together with the ribbon concordance obstruction of Zemke~\cite{Zemke} from Heegaard Floer homology. 
  To summarize the construction, take a slice knot~$C$ with maximal Thurston--Bennequin number equal to $-1$. Examples of such knots include $C_0=\overline{9_{46}}$ and $C_1=\overline{10_{140}}$ (see~\cite{knotinfo}), where $\overline{K}$ denotes the mirror image of $K$. Diagrams for $C_0$ and~$C_1$ are shown in Figure~\ref{fig:companions}.

\begin{figure}[htbp]
    \centering
    \begin{subfigure}{0.45\textwidth}
        \includegraphics[width=\textwidth]{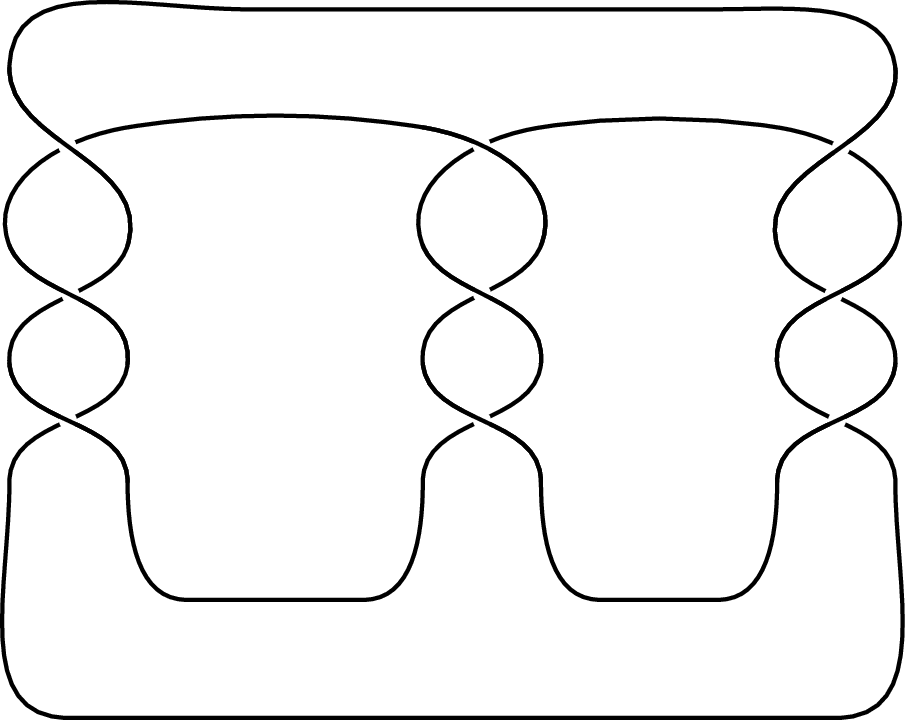}
    \end{subfigure}
    \qquad 
    \begin{subfigure}{0.45\textwidth}
        \includegraphics[width=\textwidth]{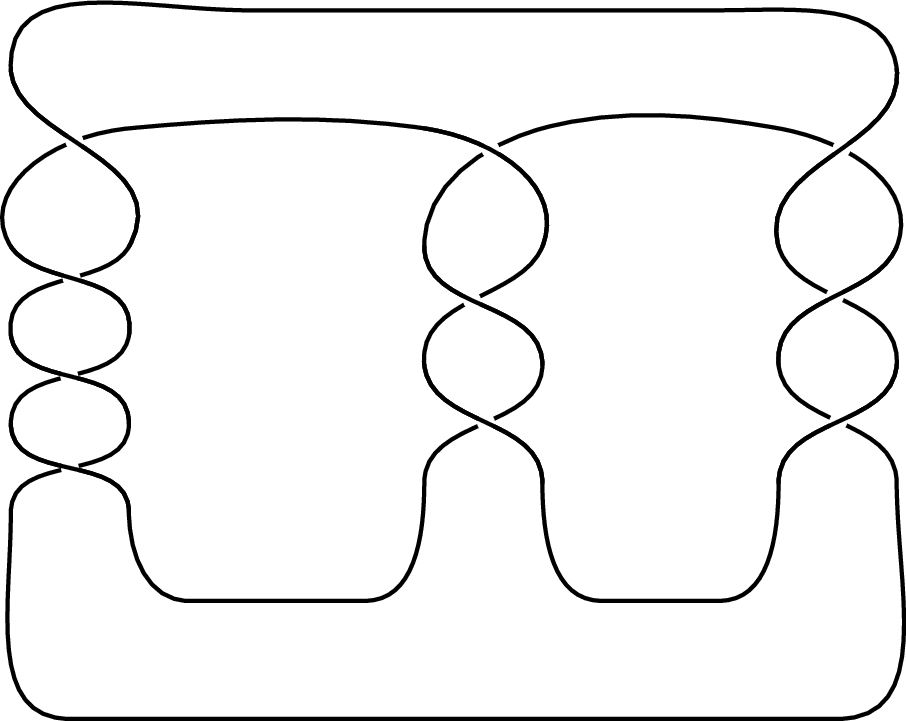}
    \end{subfigure}
    \caption{The knots $C_0=\overline{9_{46}}$ (left) and $C_1=\overline{10_{140}}$ (right).}
    \label{fig:companions}
\end{figure}

For a strongly quasipositive knot~$J$, \cite[Section 2.2]{truoel_sqp} constructs the knot $P_J(C)$, which is an instance of a satellite operation with pattern~$J$ and companion~$C$. 
Here a pattern refers to a knot $J$ that is smoothly embedded in the solid torus $V = S^1 \times D^2$. The embedding of $J$ into $V$ can be specified by drawing a meridian $\{\operatorname{pt}\} \times \partial D^2$ of $V$ in a diagram of $J$. See Figure~\ref{fig:pattern} for an example. Then the solid torus $V$ is identified with a tubular neighborhood of the companion knot $C$ in~$S^3$ via an orientation-preserving diffeomorphism  that maps $S^1 \times \{\operatorname{pt}\}$ to a Seifert longitude of $C$. The image of $J$ under this diffeomorphism is the satellite knot $P_J(C)$.
For the patterns considered in the construction of the second author, $P_J(C)$ is indeed obtained from~$J$ by tying the knot~$C$ into a specific positive band of a (strongly) quasipositive Seifert surface for~$J$ (see~\cite[Section 2.2]{truoel_sqp}).
The resulting knot is shown to be strongly quasipositive and concordant to~$J$ (see~\cite[Lemma~2]{truoel_sqp}). 

\begin{figure}[htbp]
    \centering
    \includegraphics[width=0.25\linewidth]{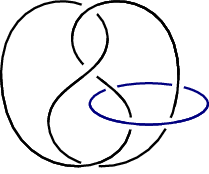}
    \caption{The right-handed trefoil $J= T_{2,3}$ (black) as pattern in the solid torus $V = S^1 \times D^2$ with the meridian of~$V$ (blue). The algebraic winding number of this pattern is $0$.}
    \label{fig:pattern}
\end{figure}

We show that for the aforementioned $C_0=\overline{9_{46}}$ and $C_1=\overline{10_{140}}$, using the right-handed trefoil $J=T_{2,3}$ as pattern (see Figure~\ref{fig:pattern}), although both are concordant to~$J$ and thus to each other, the satellite knots $K_0=P_J(C_0)$ and $K_1=P_J(C_1)$ are not ribbon concordant. 
This is done by explicit calculations. 
Diagrams for the knots $K_0$ and $K_1$ are shown in Figure~\ref{fig:knots}.

\begin{figure}[htbp]
    \centering
    \begin{subfigure}{0.48\textwidth}
        \includegraphics[width=\textwidth]{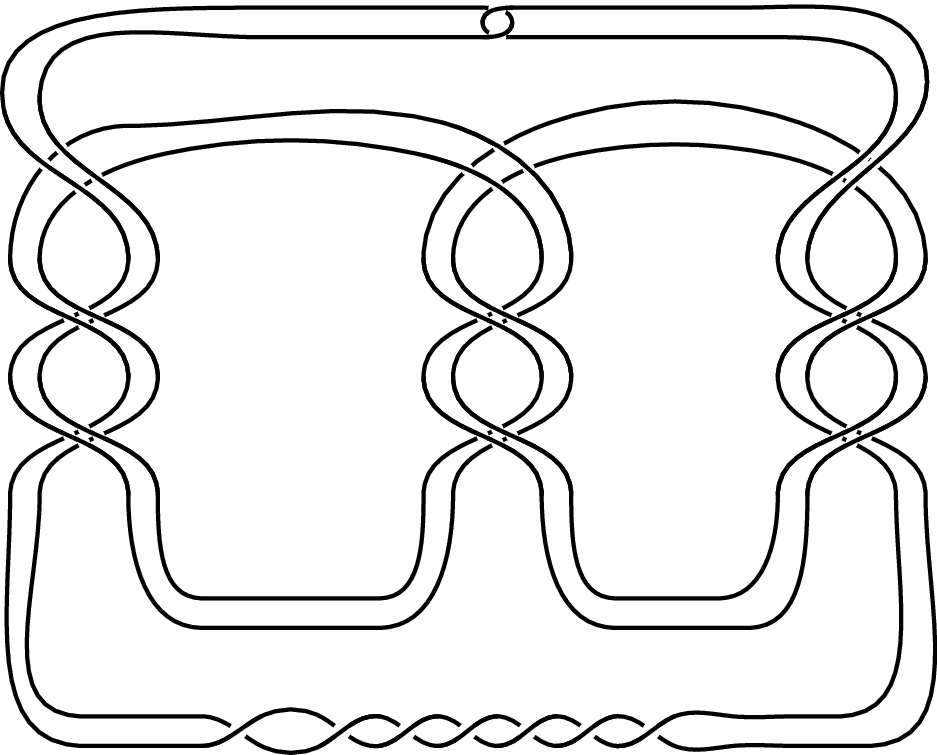}
    \end{subfigure}
    \quad 
    \begin{subfigure}{0.48\textwidth}
        \includegraphics[width=\textwidth]{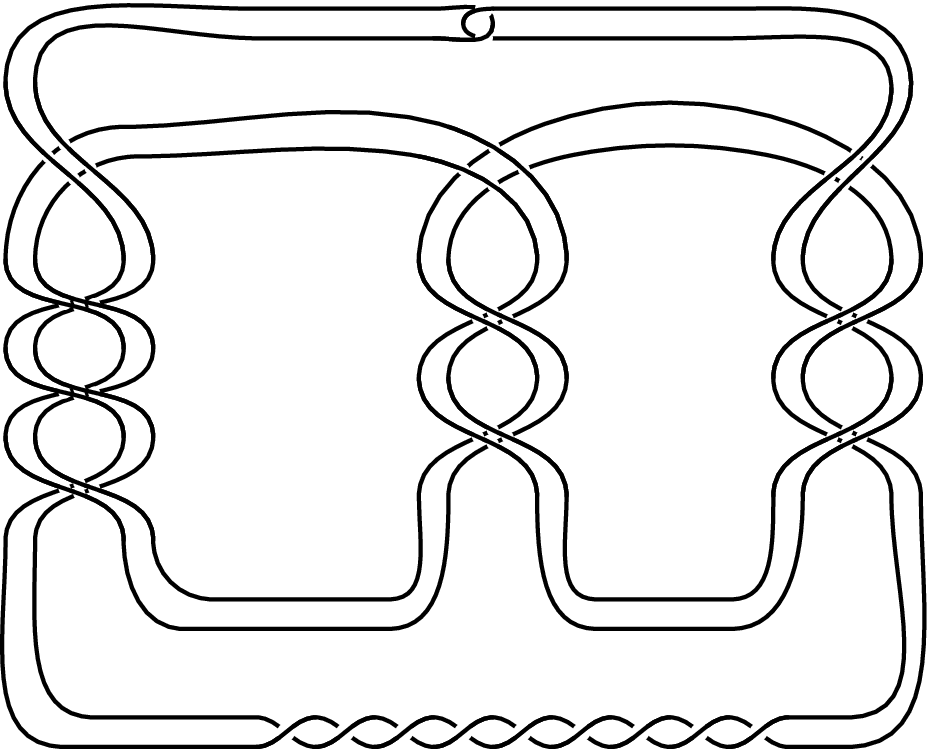}
    \end{subfigure}
    \caption{The knots $K_0=P_J(C_0)$ (left) and $K_1=P_J(C_1)$ (right), which are obtained as satellite knots with pattern~$J$ and companions $C_0=\overline{9_{46}}$ and $C_1=\overline{10_{140}}$, respectively.}
    \label{fig:knots}
\end{figure}

Given these, one can use SnapPy's \cite{SnapPy}
implementation of Szab\'o's HFK calculator \cite{HFKcalc} to compute the hat version of knot Floer homology for~$K_0$ and $K_1$ over $\F_2$.
The results are presented in Table~\ref{tab:hfk}, which shows the dimensions of $\hfk(K_i)= \bigoplus_{A,M \in \Z}\hfk_M(K_i,A)$ in the relevant $(\text{Alexander, Maslov})$ gradings for $i \in \{0,1\}$. 
\begin{table}[htbp]
\begin{tabular}{|c|c|c|} \hline & & \\[-1em]
gradings $(A,M)$ & $\dim \hfk_M(K_0,A)$ & $\dim \hfk_M(K_1,A)$ \\ \hline
(-1, -3)&0& 2 \\ \hline
  (-1, -2)& 5& 3 \\ \hline
  (-1, -1)& 4& 2 \\ \hline
  (-1, 0)& 0& 2 \\ \hline
  (0, -2)& 0& 4\\ \hline
  (0, -1)& 9& 5\\ \hline
  (0, 0)& 8&4 \\ \hline
  (0, 1)& 0& 4 \\ \hline
  (1, -1)& 0& 2 \\ \hline
  (1, 0)& 5& 3\\ \hline
  (1, 1)& 4&2\\ \hline
  (1, 2)& 0&2\\ \hline 
\end{tabular}
\caption{$\hfk$ for $K_0$ and $K_1$. 
}\label{tab:hfk}
\end{table}
The table indicates that for some gradings, such as $(0,0)$, the dimension for $K_0$ is greater than the dimension for $K_1$, while for some other gradings, such as~$(-1,-3)$, the dimension for $K_1$ is greater. Therefore, there is neither an injection $\hfk(K_0)\to \hfk(K_1)$ nor $\hfk(K_1)\to \hfk(K_0)$. According to \cite{Zemke} this implies that the two knots are not ribbon concordant.
\end{proof}

We could have also used Khovanov homology as an obstruction to ribbon concordance \cite{Levine_Zemke} in our proof of Proposition~\ref{prop:sqp_conc}. In fact, using~\cite{KnotJob} one can compute $\Kh^{i,j}(K_0)$ and $\Kh^{i,j}(K_1)$ with $\mathbb{F}_2$ coefficients and show that there is no injective map between these spaces in either direction. It is worth mentioning that both $K_0$ and $K_1$ have non-trivial Khovanov homology in negative homological degrees: $\Kh^{-2,-1}(K_0)\cong\mathbb{F}_2$, and $\Kh^{-1,0}(K_1)\cong\mathbb{F}_2$.
To the best of our knowledge, these could be the first examples of strongly quasipositive knots with nontrivial Khovanov homology in negative homological gradings. In contrast, positive links, being a narrower class of links (see \cite{Rudolph_positiveLinksSQP,nakamura}), have non-vanishing $\Kh^{i,j}$ only for $i \geq 0$, see~\cite{Khovanov_patterns,kegel}.

\subsection{Further properties of \texorpdfstring{$K_0$}{K0} and \texorpdfstring{$K_1$}{K1}}

We conclude with additional comments on our construction.

From Table~\ref{tab:hfk} we see that the graded Euler characteristic of~$\hfk$ and thus the Alexander polynomial of both knots $K_0$ and $K_1$ constructed in Section~\ref{sec:sqp} is~$t-1+t^{-1}$. This is consistent with the fact that $K_0$ and~$K_1$ are satellite knots with pattern the trefoil $J$ with algebraic winding number~$0$ (see also~Figure~\ref{fig:pattern}). Indeed, the well-known satellite formula for the Alexander polynomial predicts that $\Delta_{K_0}(t) = \Delta_{K_1}(t) = \Delta_J(t) = t-1+t^{-1}$ (see \eg \cite[Theorem 6.15]{Lickorish97}).

The same argument as in~\cite{Rudolph_surfaces} implies that the Seifert form of the two knots is equal to the Seifert form of the pattern,
that is of the trefoil. Note that SnapPy~\cite{SnapPy} computes the Seifert form correctly, providing another sanity check for our calculations.

The knots $K_0$ and $K_1$ have Seifert genus $1$, as follows from a standard argument on the behaviour of the Seifert genus under satellite operations~\cite{schubert}.
Since $K_0$ and $K_1$ are both concordant to $J=T_{2,3}$, we also have $g_4(K_0) = g_4(K_1) = g_4(J) = 1$. Alternatively, we could conclude this from the equality of the Seifert genus and $4$-genus for strongly quasipositive knots~\cite{bennequin,rudolphslice}. Thus, the examples we construct in Section~\ref{sec:sqp} have the smallest possible Seifert genus. 

Regarding Question~\ref{quest:2}, note that the knots $K_0$ and $K_1$ are not fibered by construction. This remark was already made in \cite[Remark 2.5]{pt_thesis}, based on~\cite{truoel_sqp}. One of the simplest arguments is that the algebraic winding number of the pattern is $0$, so the satellite knot is not fibered.

\section{Details of the calculations for \texorpdfstring{$K_0$}{K0} and \texorpdfstring{$K_1$}{K1}}\label{sec:details}

Since our calculations in Section~\ref{sec:sqp} were computer-based, we provide some details here. We used the diagrams in Figure~\ref{fig:knots}
to compute the PD codes of the knots $K_0$ and $K_1$. To facilitate verification of these calculations, we include Figures~\ref{fig:PD_K0} and \ref{fig:PD_K1}, which show the numbering of some of the strands used to calculate the PD codes.

\begin{figure}[htbp]
  \includegraphics[width=8cm]{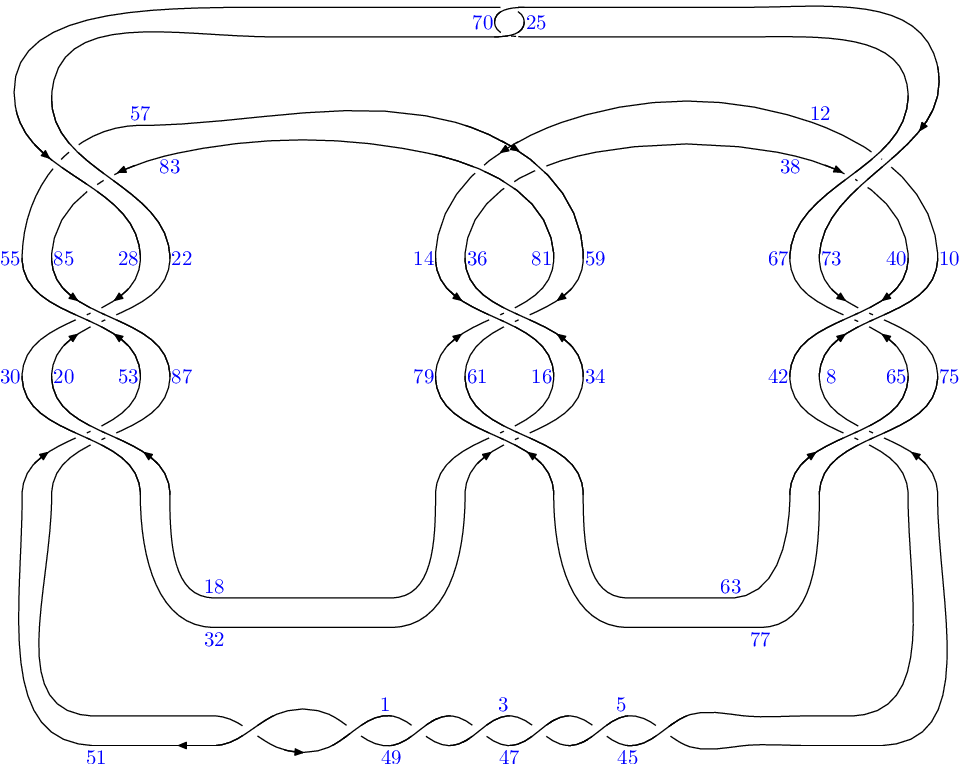}
  \caption{The diagram of $K_0$ with the strand numbering used to calculate its PD code.}\label{fig:PD_K0}
\end{figure}

\begin{figure}[htbp]
  \includegraphics[width=8cm]{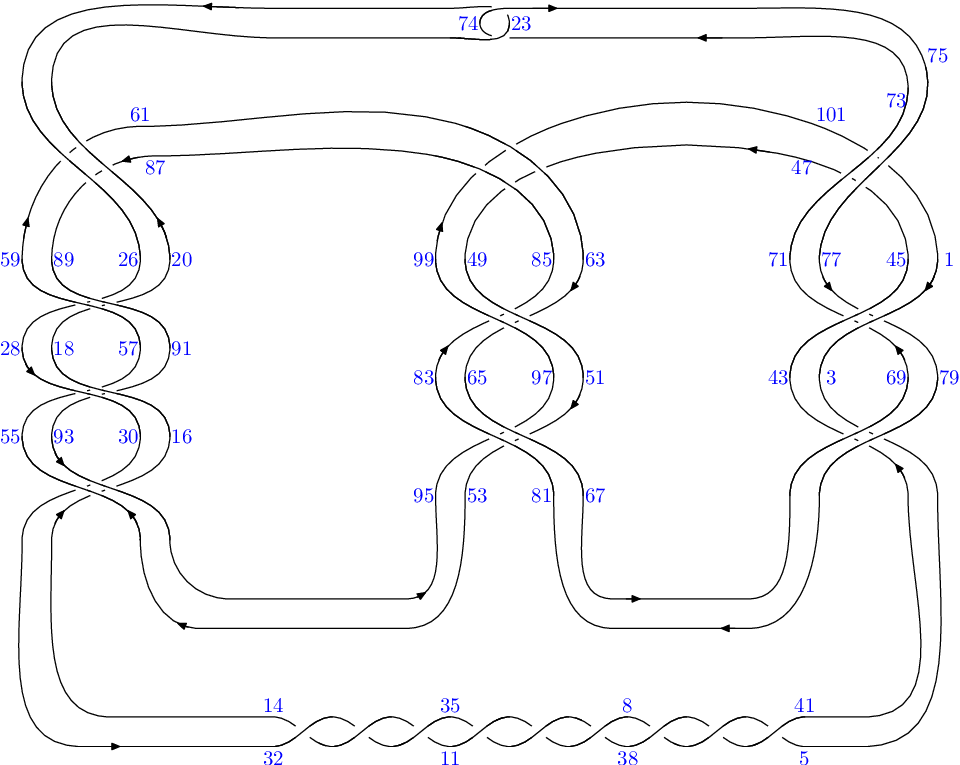}
  \caption{The diagram of $K_1$ used to calculate its PD code.}\label{fig:PD_K1}
\end{figure}

The corresponding PD codes are given as follows. 

\begin{sizeddisplay}{\footnotesize}
\begin{align*}
  K_0\colon
  [&[1,49,2,48],[47,3,48,2],[3,47,4,46],[45,5,46,4],[5,45,6,44],[43,77,44,76],\\&[6,75,7,76],[42,63,43,64],[7,65,8,64],[65,9,66,8],[66,41,67,42],\\
  &[74,9,75,10],[73,41,74,40],[39,73,40,72],[10,71,11,72],[38,67,39,68],\\
  &[11,69,12,68],[25,71,26,70],[69,25,70,24],[88,31,89,32],[51,31,52,30],\\
  &[87,19,88,18],[52,19,53,20],[20,53,21,54],[29,55,30,54],[21,87,22,86],\\
  &[28,85,29,86],[84,27,85,28],[55,27,56,26],[83,23,84,22],[56,23,57,24],\\
  &[32,77,33,78],[17,79,18,78],[33,63,34,62],[16,61,17,62],[60,15,61,16],\\
  &[79,15,80,14],[59,35,60,34],[80,35,81,36],[36,81,37,82],[37,59,38,58],\\
  &[13,83,14,82],[12,57,13,58],[89,51,90,50],[49,1,50,90]],\\
K_1\colon&
[[13, 33, 14, 32], [33, 13, 34, 12], [11, 35, 12, 34], [35, 11, 36, 10], [9, 37, 10, 36],\\
&[37, 9, 38, 8], [39, 7, 40, 6], [5, 41, 6, 40], [23, 75, 24, 74], [73, 23, 74, 22],\\
&[30, 93, 31, 94], [31, 55, 32, 54], [14, 53, 15, 54], [15, 95, 16, 94], [91, 17, 92, 16],\\
&[92, 29, 93, 30], [55, 29, 56, 28], [56, 17, 57, 18], [26, 89, 27, 90], [27, 59, 28, 58],\\
&[18, 57, 19, 58], [19, 91, 20, 90], [59, 25, 60, 24], [60, 21, 61, 22], [87, 21, 88, 20],\\
&[88, 25, 89, 26], [51, 66, 52, 67], [52, 82, 53, 81], [95, 82, 96, 83], [96, 66, 97, 65],\\
&[83, 98, 84, 99], [84, 50, 85, 49], [63, 50, 64, 51], [64, 98, 65, 97], [47, 62, 48, 63],\\
&[48, 86, 49, 85], [99, 86, 100, 87], [100, 62, 101, 61], [3, 68, 4, 69], [4, 80, 5, 79],\\
&[41, 80, 42, 81], [42, 68, 43, 67], [69, 2, 70, 3], [70, 44, 71, 43], [77, 44, 78, 45],\\
&[78, 2, 79, 1], [45, 76, 46, 77], [46, 72, 47, 71], [101, 72, 102, 73], [102, 76, 1, 75],\\
&[7, 39, 8, 38]].
\end{align*}
\end{sizeddisplay}
\normalsize
\bibliographystyle{alpha}
\bibliography{bibliography}

\end{document}